\documentclass{conm-p-l}

\usepackage{enumerate, amsmath, amsthm, amsfonts, amssymb, %xy,  
mathrsfs, graphicx, paralist, ulem}
\usepackage[usenames, dvipsnames]{color}
\usepackage{comment}

\usepackage{mathabx} %This is for \sqbullet
\usepackage{breqn} %This is for dmath
\usepackage{tikz-cd}
\numberwithin{equation}{section}
\newtheorem{Theorem}[equation]{Theorem}
\newtheorem{Proposition}[equation]{Proposition}
\newtheorem{Lemma}[equation]{Lemma}

\theoremstyle{definition}
\newtheorem{Remark}[equation]{Remark}

\newtheorem{eg}[equation]{Example}

\newtheorem{Definition}[equation]{Definition}

\newcommand{\cA}{\mathcal{A}}

\newcommand{\sD}{\mathscr{D}}

\newcommand{\cH}{\mathcal{H}}

\newcommand{\bI}{\mathbf{I}}

\newcommand{\bK}{\mathbf{K}}
\newcommand{\cK}{\mathcal{K}}

\newcommand{\sK}{\mathscr{K}}

\newcommand{\bU}{\mathbf{U}}
\newcommand{\cU}{\mathcal{U}}

\newcommand{\fg}{\mathfrak{g}}

\renewcommand{\AA}{\mathbb{A}}

\newcommand{\FF}{\mathbb{F}}
\newcommand{\GG}{\mathbb{G}}

\newcommand{\NN}{\mathbb{N}}

\newcommand{\QQ}{\mathbb{Q}}
\newcommand{\RR}{\mathbb{R}}

\newcommand{\ZZ}{\mathbb{Z}}

\renewcommand{\phi}{\varphi}
\renewcommand{\emptyset}{\varnothing}

\renewcommand{\tilde}[1]{\widetilde{#1}}

\makeatletter
\def\Ddots{\mathinner{\mkern1mu\raise\p@
\vbox{\kern7\p@\hbox{.}}\mkern2mu
\raise4\p@\hbox{.}\mkern2mu\raise7\p@\hbox{.}\mkern1mu}}
\makeatother

\DeclareMathOperator{\sgn}{sgn}

\newcommand{\Inv}{\mathrm{Inv}}

\newcommand{\suchthat}{\mid}

\newcommand{\wt}{\text{wt}}

\newcommand{\textif}{\text{ if }}
\newcommand{\textand}{\text{ }\mathrm{and}\text{ }}

\newcommand{\dir}{\mathrm{dir}}
\newcommand{\ep}{\mathrm{end}}
\newcommand{\mathperiod}{\, .}

\renewcommand{\comment}[1]{}

\begin{document}

\title[Walk algebras and distinguished subexpressions]{Walk algebras, distinguished subexpressions, and point counting in Kac-Moody flag varieties}
\author{Dinakar Muthiah}
\thanks{D.M. was partially supported by a PIMS postdoctoral fellowship}
\address{
Department of Mathematics and Statistics, 
Lederle Graduate Research Tower, 1623D,
University of Massachusetts Amherst,
710 N. Pleasant Street,
Amherst, MA 01003-9305, USA
}
\email{muthiah@math.umass.edu}
\author{Daniel Orr}
\thanks{D.O. was partially supported by NSF grant DMS-1600653}
\address{Department of Mathematics (MC 0123), 460 McBryde Hall, Virginia Tech, 225 Stanger Street, Blacksburg, VA 24061 USA}
\email{dorr@vt.edu}
\maketitle

\begin{abstract}
  We study walk algebras and Hecke algebras for Kac-Moody root systems. Each choice of orientation for the set of real roots gives rise to a corresponding ``oriented'' basis for each of these algebras.
We show that the notion of distinguished subexpression naturally arises when studying the transition matrix between oriented bases.
We then relate these notions to the geometry of Kac-Moody flag varieties and Bott-Samelson varieties. In particular, we show that the number of points over a finite field in certain intersections of these varieties is given by change of basis coefficients between oriented bases of the Hecke algebra. Using these results we give streamlined derivations of Deodhar's formula for $R$-polynomials and point-counting formulas for specializations of nonsymmetric Macdonald polynomials $E_\lambda(\mathsf{q},t)$ at $\mathsf{q}=0,\infty$. 
\end{abstract}

%\tableofcontents

\section{Introduction}

Alcove walk algebras were introduced in \cite{Ram} to study the combinatorics of affine Hecke algebras and Hall-Littlewood polynomials. The purpose of this paper is to study walk algebras for general Kac-Moody groups. Roughly speaking, the walk algebra is obtained from the Hecke algebra by forgetting the braid relations and weakening the quadratic relations; thus the Hecke algebra is naturally a quotient of the walk algebra.

For each orientation of the set of real roots we define a basis of the walk algebra indexed by subexpressions. Our first aim is to investigate the combinatorics of these bases. We show that the change-of-basis matrix between oriented bases is controlled by a subset of distinguished subexpressions depending on the chosen pair of orientations (Theorem~\ref{thm:change-of-basis-formula}). The notion of distinguished subexpression (for certain choices of orientations) originated in the work of Deodhar \cite{Deo} and has also played an important role in work of Billig and Dyer \cite{BD}. It also appears in the works of Gaussent and Littelmann \cite{Gaussent-Littelmann} and Schwer \cite{Sch} (as positively folded galleries); and in the works of Ram \cite{Ram} and Parkinson, Ram, and Schwer \cite{PRS} (as positively folded alcove walks). We demonstrate, generalizing \cite{Ram}, that the definition of distinguished subexpressions arises naturally as a consequence of the relations in the walk algebra.

After establishing the basic result on change of oriented bases in walk algebras, we descend to the Hecke algebra and relate the change of basis coefficients to point counting in Kac-Moody flag varieties (Theorem~\ref{thm:point-count-in-flag-variety}). Up to a simple normalization, the change of basis coefficients from the standard orientation to another orientation are given by the number of points over a finite field in certain orbit intersections in the flag variety. This equality is already known as a consequence of \cite[Theorem 5]{BD}, but our approach is different. We mention that oriented bases of Hecke algebras appeared in earlier works of Dyer \cite{Dy,Dy2}. More recently, work of Gobet \cite{Gobet} establishes the positivity of the expansion coefficients of the Kazhdan-Lusztig basis of the Hecke algebra into any fixed oriented basis.

Our primary goal in this paper is to provide a synthesized exposition of various results and ideas from \cite{BD}, \cite{Gaussent-Littelmann}, \cite{Ram}, and \cite{PRS}.
In addition to our investigation of walk algebras, one novelty of our presentation is that we give a different, more geometric proof of the point-counting result described above. This proof, which is inspired by the methods of \cite{Gaussent-Littelmann}, uses the Bott-Samelson resolution of Schubert varieties. We believe that for certain audiences this proof may be more approachable than existing proofs which involve direct manipulation of relations in the Kac-Moody group.

We use these results to give streamlined derivations of Deodhar's formula for $R$-polynomials \cite{Deo} and point-counting formulas for nonsymmetric Macdonald polynomials $E_\lambda(\mathsf{q},t)$ at $\mathsf{q}=0,\infty$. In both of these situations we emphasize the role played by change of oriented bases. We discuss how the point-counting formula for nonsymmetric Macdonald polynomials at $\mathsf{q}=0$ relates to a similar result of \cite{Ion}.

In future work we aim to adapt the constructions of this paper to the setting of $p$-adic Kac-Moody groups and their Iwahori-Hecke algebras (see \cite{BKP,BGR,M} for existing work on this topic). One promising direction is to incorporate the work of Gaussent and Rousseau (\cite{GR1,GR2}) into the present exposition. Their work suggests that the notion of expression, which is a sequence of Weyl group elements indexed by integers, should be replaced with a certain type of continuous path in the coweight space indexed by the interval $[0,1] \subseteq \RR$. For example, they work with a notion called ``Hecke path'' that appears to be a useful proxy for the notion of positively folded alcove path.

Another promising direction is to incorporate the combinatorial study of the object that replaces the Weyl group. This object is no longer a Coxeter group, but it carries a natural Bruhat order, introduced in \cite{BKP}, and a length function \cite{M,MO} compatible with this order. Conjectures about point-counting in double-affine flag varieties were made using this length function in \cite[\S8.2]{MO}. A full generalization of the results of this note would include answers to these conjectures. With such a goal in mind, one of our initial motivations for writing this note was to give an exposition that is likely to generalize to the $p$-adic Kac-Moody setting.

\subsection*{Acknowledgments}

D.M. thanks Abid Ali, Gerald Cliff, Manish Patnaik, and Anna Pusk\'as for their talks and discussion in the learning seminar on the topic of \cite{Gaussent-Littelmann} which took place in Fall 2015 at the University of Alberta. He especially thanks Manish Patnaik and Anna Pusk\'as for their work in organizing the seminar. D.O. thanks Elizabeth Mili\'{c}evi\'{c} for helpful discussions about point counting in affine flag varieties. We also wish to thank an anonymous referee whose comments led to several improvements in our exposition.

\section{Notation and preliminaries}
\label{S:notation}

Let us consider a Kac-Moody root system with index set $I$, Weyl group $W$, simple roots $\Pi = \{ \alpha_{i} \}_{i \in I}$, and simple reflections $S = \{s_{i} \}_{i \in I}$. For $ s_i \in S$, we will abuse notation and sometimes write $\alpha_{s_i} = \alpha_{i}$.

Imaginary roots will not play a role in what we do, so when we say root we mean {\it real root}. Let us write $\Delta_+$ (resp. $\Delta_-$) for the set of positive (resp. negative) roots.  Let $\ell: W \rightarrow \NN$ be the length function.

For a positive integer $r$, we write $[r]=\{1,2,\dotsc,r\}$.

\subsection{Expressions and subexpressions}

An {\it expression} of length $r$ is a $r$-tuple $\sigma = (\sigma_1, \dotsc, \sigma_r) \in S^r$. For any such $\sigma$ there is a unique $r$-tuple $(i_1, \dotsc, i_r) \in I^r$ such that $\sigma = (s_{i_1}, \dotsc, s_{i_r})$.

A {\it subexpression} $\tau$ of $\sigma$ is a $r$-tuple $\tau = (\tau_1, \dotsc, \tau_r) \in S^r$ such that $\tau_j \in \{ e, \sigma_j \}$. Note $\tau$ consists of the data of the $r$-tuple  $(\tau_1, \dotsc, \tau_r)$ {\it and} the expression $\sigma$, i.e., $\tau$ remembers of which expression it is a subexpression. Notationally, let us write $\tau \leq \sigma$ to indicate that $\tau$ is a subexpression of $\sigma$. Let us also write $\sigma \leq \sigma$ for the {\it full subexpression} that agrees with $\sigma$ in every slot.

Given two subexpressions $\tau, \rho$ of $\sigma$, we say that $\rho$ is a subexpression of $\tau$, denoted $\rho \leq \tau \leq \sigma$, if $\tau_k = e $ implies $\rho_k = e$. 
Given a subexpression $\tau = (\tau_1, \dotsc, \tau_r)$ let us define the {\it product} of $\tau$ by
\begin{align}
  \label{eq:2}
  \pi(\tau) = \tau_1 \cdots \tau_r
  \mathperiod 
\end{align}

\subsection{Orientations}

An {\it orientation} $A$ is a decomposition of $\Delta = \Delta^A_+ \cup \Delta^{A}_-$ such that $\Delta^A_+ \cap \Delta^{A}_- = \emptyset$, $\Delta^A_{+} =  -\Delta^{A}_{-}$, and $\Delta^A_+$ is convex: 
\begin{align}\label{E:convex}
\text{$\alpha,\beta\in\Delta^A_+$ and $a\alpha+b\beta\in\Delta$ for $a,b>0$ implies $a\alpha+b\beta\in\Delta_A^+$.}
\end{align} 
Given a root $\beta$, we will write $\beta >_{A} 0$ if $\beta \in \Delta^A_{+}$ (and $\beta <_{A} 0$ if $\beta \in \Delta^A_{-}$). Let us write $A_0$ for the {\it standard orientation} where $\Delta^{A_0}_+ = \Delta_+$.

\begin{eg}
The Weyl group acts on the set of orientations as follows. Given an orientation $A$ and $x \in W$, we can form a new orientation $x(A)$ by defining $\beta >_{x(A)} 0$ if and only if $x^{-1}(\beta) >_{A} 0$.
\end{eg}

\begin{eg}
For any orientation $A$, its negative $-A$ is the orientation with $\Delta^{-A}_+=-\Delta^A_+$.
\end{eg}

Given an orientation and a root $\beta$, we define the sign of $\beta$ with respect to $A$ by
\begin{align}
  \label{eq:9}
  \sgn_{A}(\beta) =
  \begin{cases}
    +1 & \textif \beta >_{A} 0 \\
    -1 & \textif \beta <_{A} 0  
    \mathperiod
  \end{cases}
\end{align}

\begin{eg}\label{eg:affine}
Suppose that $\Delta=\{\alpha+n\delta \suchthat \alpha\in\Delta_\circ,\, n\in\ZZ\}$ is an irreducible root system of untwisted affine Kac-Moody type. Here $\Delta_\circ$ is a Kac-Moody root system of finite-type and the positive roots $\Delta_+$ are those $\alpha+n\delta$ such that $n\geq 0$ and $n>0$ if $\alpha\in(\Delta_\circ)_-$. Another important orientation is the {\it periodic orientation} $A_\infty$, which is defined by $\alpha+n\delta>_{A_\infty}0$ if and only if $\alpha\in(\Delta_\circ)_+$. We also define $A_{-\infty}=-A_\infty$. We note that the orientations $A_\infty$ are not of the form $x(A_0)$ for any $x\in W$.
\end{eg}

\begin{Remark}
In \cite[Introduction]{BD}, one finds some remarks on the classification of orientations. Here we simply mention that in finite Kac-Moody type, all orientations are of the form $x(A_0)$ for $x\in W$, while in affine Kac-Moody type, there are finitely many orientations up to the action of the Weyl group.
\end{Remark}

\begin{Remark}
Let $\widehat{\Delta}\supset\Delta$ be the set of all real and imaginary roots of the given Kac-Moody root system. One may define an orientation of $\widehat{\Delta}$ as above, but with the convexity assumption relaxed by replacing $a\alpha+b\beta$ in \eqref{E:convex} simply by $\alpha+\beta$. The intersection of such an orientation of $\widehat{\Delta}$ with $\Delta$ produces an orientation of $\Delta$ in the sense defined above (using \eqref{E:convex}). In fact, it follows from the proof of \cite[Proposition~1]{BD} that the orientations of $\Delta$ obtained in this way are exactly those of the form $\Delta^A_+=\Delta\cap P$, where $P$ is the cone of elements greater than or equal to $0$ with respect to some lexicographic ordering of the $\mathbb{R}$-span of $\Delta$. The latter are known as {\it geometric orientations} of $\Delta$, which form in general a proper subset of all orientations of $\Delta$. However, the applications considered in this paper involve only geometric orientations.
\end{Remark}

\subsection{The length function}

The following discussion of the length function will be used when we relate the change-of-basis coefficients for oriented bases to point-counting in Kac-Moody flag varieties. 

Following Dyer \cite{Dy}, we define for any orientation $A$
% a partial order $\leq_A$ on $W$ called the $A$-twisted Bruhat order\footnote{The twisted Bruhat order that we define here is a ``right-handed'' version of Dyer's.} as the partial order generated by the elementary relations
% \begin{align}
%   w <_A ws_\alpha, \quad\text{for all}\quad \alpha\in\Delta_+\cap w^{-1}(\Delta_+^A)
%   \mathperiod
% \end{align}
% There is a compatible
a length function $\ell_A: W\to \ZZ$ defined by
\begin{align}
\ell_A(w)&=\sum_{\alpha\in\Inv(w^{-1})} \sgn_A(\alpha) =\ell(w)-2\#(\Inv(w^{-1})\cap\Delta^A_-)
\end{align}
where $\Inv(w^{-1})=\Delta_+\cap w(\Delta_-)$. It satisfies the following basic properties.
%By ``compatible'' we mean the following
% \begin{Lemma}\label{L:ellA-wsi}
% 	For any $w\in W$ and $i\in I$,
% 	\begin{align}
% 	\comment{
% 		\ell_A(s_iw)&=\ell_A(w)+
% 		\begin{cases}
% 		+1 & \text{if $\alpha_i\in \Delta^A_+ \cap w(\Delta_+)$}\\
% 		-1 & \text{if $\alpha_i\in \Delta^A_- \cap w(\Delta_+)$}\\
% 		-1 & \text{if $\alpha_i\in \Delta^A_+ \cap w(\Delta_-)$}\\
% 		+1 & \text{if $\alpha_i\in \Delta^A_- \cap w(\Delta_-)$}
% 		\end{cases}\\
% 	}
% 	\label{E:ellA-wsi}
% 	\ell_A(ws_i)-\ell_A(w)&=
% 	\begin{cases}
% 	+1 & \text{if $w(\alpha_i)>_A 0$}\\
% 	-1 & \text{if $w(\alpha_i)<_A 0$}
%         \mathperiod
% 	\end{cases}
% 	\end{align}
% \end{Lemma}
\begin{Lemma}\label{L:ell-A}
	Let $A$ be an orientation. Then:
	\begin{enumerate}
        \item For $i \in I$, we have
                  \begin{align}
  	\label{E:ellA-wsi}
	\ell_A(ws_i)-\ell_A(w)&=
	\begin{cases}
	+1 & \text{if $w(\alpha_i)>_A 0$}\\
	-1 & \text{if $w(\alpha_i)<_A 0$}
        \mathperiod
	\end{cases}
                  \end{align}
          
		\item $\ell_{A}(e)=0$.
		\item For any $w\in W$, $\ell_{-A}(w)=-\ell_A(w)$.
		\item For any $x,w\in W$, $\ell_{x(A)}(w)=\ell_A(x^{-1}w)-\ell_A(x^{-1})$.
	\end{enumerate}
\end{Lemma}

For an orientation $A$ and an expression $\tau=(\tau_1,\dotsc,\tau_r)$, we define
\begin{align}
\label{E:zeta+}
\zeta_\tau^+ &= \#\{k \in [r] \suchthat \tau_k = \sigma_k \textand -\tau_1 \cdots \tau_k (\alpha_{\sigma_k}) >_A 0 \}\\
\label{E:zeta-}
\zeta_\tau^- &= \#\{k \in [r] \suchthat \tau_k = \sigma_k \textand -\tau_1 \cdots \tau_k (\alpha_{\sigma_k}) <_A 0 \}\\
\label{E:kappa}
  \kappa_\tau &= \#\{ k \in [r] \suchthat \tau_k = e \}
\mathperiod
\end{align}

\begin{Lemma}\label{L:kappa+2tau}
	Let $A$ be an orientation and let $\sigma$ be a reduced expression. Then for any subexpression $\tau\le\sigma$, we have
	\begin{align}\label{E:kappa+2tau}
          \kappa_\tau+2\zeta_\tau^+ &= \ell(\pi(\sigma))+\ell_A(\pi(\tau))
                                      \mathperiod
	\end{align}
\end{Lemma}

\begin{proof}
	Applying \eqref{E:ellA-wsi} successively to the Weyl group elements in the sequence $e, \tau_1, \tau_1\tau_2, \dotsc$, we deduce that $\ell_A(\pi(\tau))=\zeta_\tau^+-\zeta_\tau^-$. We add this to $\ell(\pi(\sigma))=\zeta_\tau^++\zeta_\tau^-+\kappa_\tau$ to obtain \eqref{E:kappa+2tau}.
\end{proof}

\begin{Remark}
	When $A=-A_0$, formula \eqref{E:kappa+2tau} reduces to $\kappa_\tau+2\zeta_\tau^+=\ell(\pi(\sigma))-\ell(\pi(\tau))$. This is \cite[Lemma 5.1]{Deo}.
\end{Remark}

\subsubsection{Length for the periodic orientation in affine type}

Suppose that $W=Q^\vee_\circ\rtimes W_\circ$ is the Weyl group of an irreducible Kac-Moody root system of untwisted affine type. Here $W_\circ$ and $Q^\vee_\circ$ are the Weyl group and coroot lattice of a finite type Kac-Moody algebra $\fg_\circ$. Let $2\rho$ be the sum of all positive roots of $\fg_\circ$.

Let us denote an arbitrary element $x=(\lambda,u)\in W$ by $\varpi^\lambda u$. For such $x$ we write $\wt(x)=\lambda\in Q_\circ^\vee$ and $\dir(x)=u\in W_\circ$. The elements $\varpi^\lambda=\varpi^\lambda e\in W$ for $\lambda\in Q^\vee_\circ$ pairwise commute and are called translation elements.
\begin{Lemma}\label{L:A-inf-length}
	For $W=Q^\vee_\circ\rtimes W_\circ$, we have
	\begin{align}
	\label{E:A-inf-length}
	\ell_{A_\infty}(\varpi^\lambda u) &=-\langle\lambda,2\rho\rangle+\ell(u)\\
	\label{E:A-minf-length}
          \ell_{A_{-\infty}}(\varpi^\lambda u) &=\langle\lambda,2\rho\rangle-\ell(u)
        \mathperiod
	\end{align}
\end{Lemma}

\begin{proof}
	The two formulas are equivalent, so let us prove the second formula. Let $\Delta=\{\alpha+r\delta\}$ be the set of affine roots as in Example~\ref{eg:affine}. We have the following formula for the action of $u^{-1}\varpi^{-\lambda}$ on affine roots
	\begin{align}
	u^{-1}\varpi^{-\lambda}(\alpha+r\delta)&=u^{-1}(\alpha)+(r+\langle\lambda,\alpha\rangle)\delta
        \mathperiod                                                 
	\end{align}
	Thus $\alpha+r\delta\in\Inv(u^{-1}\varpi^{-\lambda})$ if and only if one of the following holds:
	\begin{align}
	&\text{$\alpha>0$, $u^{-1}(\alpha)>0$, and $0\leq r < -\langle\lambda,\alpha\rangle$}\\
	&\text{$\alpha>0$, $u^{-1}(\alpha)<0$, and $0\leq r \leq -\langle\lambda,\alpha\rangle$}\\
	&\text{$\alpha<0$, $u^{-1}(\alpha)>0$, and $0< r < -\langle\lambda,\alpha\rangle$}\\
	&\text{$\alpha<0$, $u^{-1}(\alpha)<0$, and $0< r \leq -\langle\lambda,\alpha\rangle$}
        \mathperiod
	\end{align}
	Therefore,
	\begin{align}
	\ell_{A_{-\infty}}(\varpi^\lambda u)&=\sum_{\substack{\alpha\in\Delta_\circ\\-\langle\lambda,\alpha\rangle>0}}-\sgn(\alpha)\left(-\langle\lambda,\alpha\rangle+\begin{cases}1 &\text{if $\alpha>0$ and $u^{-1}(\alpha)<0$}\\-1&\text{if $\alpha<0$ and $u^{-1}(\alpha)>0$}\end{cases}\right)\\
	&\qquad-\#\{\alpha\in\Delta_\circ : \text{$\alpha>0$, $u^{-1}(\alpha)<0$, and $\langle\lambda,\alpha\rangle=0$}\}\notag\\
                                            &=\langle\lambda,2\rho\rangle-\ell(u)
        \mathperiod
        \qedhere
	\end{align} 
\end{proof}

\section{Hecke algebras and walk algebras}

\subsection{Walk algebras}

Following \cite{Ram}, we define the {\it walk algebra} $\cA$ to be the unital $\ZZ$-algebra generated by symbols $\{ c_i^+, c_i^-, f_i^+, f_i^- \}_{i \in I}$ subject to the following relations for each $i \in I$
\begin{align}
  \label{eq:1}
 c^+_i  &= c^-_i + f^+_i \\
  \label{eq:1'}
  c^-_i  &= c^+_i + f^-_i
\mathperiod
\end{align}
We note that for all $i\in I$
\begin{align}
\label{eq:1''}
f_i^-=-f_i^+
\mathperiod
\end{align}

\begin{Remark}
	In affine Kac-Moody type, \cite{Ram} interprets the algebra $\cA$ in terms of alcove walks. The $c^{\pm}_i$ (resp. the $f^{\pm}_i$) generators correspond to crossing (resp. folding along) the $i$-wall of an alcove. For this reason, in this case $\cA$ is called the {\it alcove} walk algebra in \cite{Ram}.
\end{Remark}

\subsection{Hecke algebras}

Let $\cH$ be the Hecke algebra corresponding to $W$. Explicitly, $\cH$ is the unital $\ZZ[v,v^{-1}]$-algebra generated by symbols $\{T_i\}_{i \in I}$ subject to the braid relations and the following quadratic relations
\begin{align}
  \label{eq:3}
  (T_i - v) ( T_i + v^{-1}) = 0, \quad i\in I
  \mathperiod
\end{align}
We immediately have $T_i^{-1} = T_i - (v - v^{-1})$.
For each $w \in W$, we define
\begin{align}
  \label{eq:4}
 T_w = T_{i_1} \cdots T_{i_r} 
\end{align}
where $\sigma = (s_{i_1}, \dotsc, s_{i_r})$ is a reduced expression for $w$.
Following \cite{Ram}, we have the following.
\begin{Proposition}
There is a surjective $\ZZ$-algebra map  
\begin{align}
  \label{eq:5}
  \Phi: \cA \rightarrow \cH
\end{align}
given by
\begin{align}
  \label{eq:6}
  c_i^{\pm} &\mapsto T_i^{\pm 1} \\
  f_i^{\pm} &\mapsto \pm (v - v^{-1})
\end{align}
for all $i \in I$.
\end{Proposition}

Ram explicitly describes the kernel of this map in the affine case, but since it is not necessary for our purposes we will not reproduce it here.

\subsection{Oriented bases}

Fix an orientation $A$. Let $\sigma$ be an expression of length $r$, and let $\tau$ be a subexpression. We define 
\begin{align}
L(\tau \leq \sigma, A) = x_{1} \cdots x_{r}\in\cA
\end{align}
where $x_{k} = c_{\sigma_k}^{\epsilon_k}$ if $\tau_k = \sigma_k$ and $x_{k} = f_{\sigma_k}^{\epsilon_k}$ if $\tau_k = e$; the sign $\epsilon_k$ is defined by
\begin{align}
  \epsilon_k = \sgn_{A}( -\tau_1 \cdots \tau_{k} (\alpha_{\sigma_k}) )
  \mathperiod
\end{align}
We also define
\begin{align}
  \label{eq:33}
  T(\tau \leq \sigma, A) = \Phi\left( L(\tau \leq \sigma, A) \right)\in\cH
  \mathperiod
\end{align}
To simplify notation, we will write $L(\sigma \leq \sigma, A) = L(\sigma, A)$ and $T(\sigma \leq \sigma, A) = T(\sigma, A)$.

\begin{Remark}\label{R:alcove-walks}
  In affine Kac-Moody type, expressions $\sigma$ exactly correspond to ``unfolded alcove walks'', and subexpressions $\tau \leq \sigma$ correspond to foldings of the alcove walk $\sigma$. Ram considers the orientation $A_{-\infty}$, and in Ram's terminology, $L(\tau \leq \sigma, A_{-\infty})$ is a word in $\cA$ corresponding to the ``alcove walk'' $\tau \leq \sigma$. Our definition is the natural generalization of his to arbitrary orientations.
\end{Remark}

Immediately from the definition of $L(\tau\leq\sigma,A)$, we have the following recurrence relation for these elements:
\begin{Proposition}
Fix an orientation $A$. Let $\sigma = (\sigma_1, \dotsc, \sigma_r)$ be an expression of length $r$, and let $\tau = (\tau_1, \dotsc, \tau_r)$ be a subexpression of $\sigma$. Then we have: 
  \label{prop:simple-L-recursion}
  \begin{align}
    \label{eq:13}
    L(\tau \leq \sigma, A) =  L(\tau_1 \leq \sigma_1, A) L((\tau_2, \dotsc, \tau_r) \leq (\sigma_2, \dotsc, \sigma_r),  \tau_1^{-1}(A))
    \mathperiod
  \end{align}
\end{Proposition}

The following is \cite[Lemma 3.1]{Ram} for the periodic orientation (see also \cite[Lemma 2.4.2]{Goertz}), and it holds in general by the same argument.
\begin{Proposition}
  The set of $\{ L(\tau \leq \sigma, A) \}$ form a basis of $\cA$ as we vary over all expressions $\sigma$ and subexpressions $\tau \leq \sigma$.
\end{Proposition}
\noindent We will call a basis of $\cA$ of this form an {\it oriented basis}.
\begin{Remark}
  In \cite[Theorem 3.1.1]{Goertz}, it is shown that $T(\sigma, A)$ depends only on $\pi(\sigma)$ for any expression $\sigma$.
%  In other words, any two expressions with the same product map to the same element of $\cH$, regardless of orientation (which is fixed).
  For any $w\in W$ we may therefore unambiguously define $T(w, A)\in\cH$ to be $T(\sigma,A)$ for any expression $\sigma$ that $\pi(\sigma)=w$. Correspondingly, for any fixed $A$, the set $\{T(w,A) \suchthat w\in W\}$ is a basis for $\cH$, which we call an oriented basis. These oriented bases of the Hecke algebra $\cH$ appeared (with slightly different conventions) in earlier work of Dyer \cite{Dy} (see also \cite[\S9]{Dy2}).
\end{Remark}

\section{Change of basis formula}

In this section we study the change of basis between oriented bases of $\cA$. This gives rise to the notion of distinguished subexpressions.

%\subsection{Distinguished subexpressions}

\begin{Definition}
	Let $A$, $A'$ be orientations. Let $\sigma$ be an expression of length $r$, let $\tau$ be a subexpression of $\sigma$, and let $\rho$ be a subexpression of $\tau$. Let $K = \{ k \in [r] \suchthat \rho_k = e, \tau_k = \sigma_k \}$. Then we say $\rho$ is a {\it distinguished subexpression of $\tau$ (with respect to $A$ and $A'$)} if for all $k \in K$, we have
	\begin{align}
	\label{eq:7}
          \sgn_{A}( -\tau_1 \cdots \tau_k(\alpha_{\sigma_k}) ) = \sgn_{A'}(-\rho_1 \cdots \rho_k(\alpha_{\sigma_k}))
          \mathperiod
	\end{align}
	Let us write $\sD^{A'}_A(\tau \leq \sigma)$ for the set of $\rho$ satisfying the above conditions.
\end{Definition}

\begin{Remark}
	In the setting of Remark~\ref{R:alcove-walks}, distinguished subexpressions for $A=A_0$, $A'=A_{-\infty}$, and $\tau=\sigma$, correspond to ``positively folded alcove walks''.
\end{Remark}

\begin{Definition}
	Let $A$, $A'$ be orientations. Let $\sigma$ be an expression of length $r$, let $\tau$ be a subexpression of $\sigma$, and let $\rho$ be a subexpression of $\tau$. We define
	\begin{align}
	\label{eq:8}
          &\cK^{A'}_{A}(\rho\leq\tau\leq\sigma) \\
          \notag
          &= \{ k \in [r] \suchthat \rho_k = e, \tau_k = e, \sgn_{A}(- \tau_1 \cdots \tau_k( \alpha_{\sigma_k}) ) \neq \sgn_{A'}(- \rho_1 \cdots \rho_k( \alpha_{\sigma_k}) ) \}
         \mathperiod
	\end{align}

\end{Definition}

%\subsection{Change of basis}

\begin{Theorem}
  \label{thm:change-of-basis-formula}
Let $A$ and $A'$ be any orientations, and let $\tau$ be a subexpression of an expression $\sigma$. Then we have
\begin{align}
L(\tau\leq\sigma, A) = \sum_{\rho \in \sD^{A'}_A\left(\tau\leq\sigma\right)} (-1)^{\# \sK_A^{A'}(\rho\leq\tau\leq\sigma)}L(\rho \leq \sigma, A')
\mathperiod
\end{align}
\end{Theorem}

This is best understood in terms of repeated application of \eqref{eq:1}, \eqref{eq:1'}, and \eqref{eq:1''} to ``straighten'' each factor in $L(\tau\leq\sigma,A)$ from left to right. Let us demonstrate this with an example before proving Theorem~\ref{thm:change-of-basis-formula}.

\begin{eg}
	For $W=\langle s_1,s_2\rangle$, the Weyl group of $SL_3$, consider $A=s_1(A_0)$, $A'=s_2s_1(A_0)$, and $\tau=(s_1,s_2,e)\leq (s_1,s_2,s_1)=\sigma$. 
	We have
	\begin{align*}
	   &L(\tau\leq\sigma,A)\\&=c_1^-c_2^+f_1^-\\
	              &=(c_1^++f_1^-)c_2^+f_1^-\\
	              &=c_1^+c_2^+f_1^-+f_1^-c_2^+f_1^-\\
	              &=c_1^+(c_2^-+f_2^+)f_1^-+f_1^-(c_2^-+f_2^+)f_1^-\\
	              &=-c_1^+c_2^-f_1^+-c_1^+f_2^+f_1^+-f_1^-c_2^-f_1^++f_1^-f_2^+f_1^-\\
                  &=-L((s_1,s_2,e),A')-L((s_1,e,e),A')-L((e,s_2,e),A')+L((e,e,e),A')
                                \mathperiod
	\end{align*}
	In this example all subexpressions of $\tau$ are distinguished.
\end{eg}

\newcommand{\tsigma}{\tilde{\sigma}}
\newcommand{\ttau}{\tilde{\tau}}
\newcommand{\trho}{\tilde{\rho}}

The proof of Theorem~\ref{thm:change-of-basis-formula} is based on the following, which is straightforward.
\begin{Proposition}
  \label{prop:L-recursion}
Let $A$ and $A'$ be any orientations, and let $\tau$ be a subexpression of an expression $\sigma$ of length $r$. 
Let $\tsigma = (\sigma_2, \dotsc, \sigma_r)$ and $\ttau = (\tau_2, \dotsc, \tau_r)$. Consider three cases for this data:
\begin{align*}
&\text{{\rm Case A:} \quad $\sgn_{A}(\tau_1(\alpha_{\sigma_1})) = \sgn_{A'}(\tau_1(\alpha_{\sigma_1}))$}\\
&\text{{\rm Case B:} \quad $\sgn_{A}(\tau_1(\alpha_{\sigma_1})) \neq \sgn_{A'}(\tau_1(\alpha_{\sigma_1}))  \textand  \tau_1 = e$}\\
&\text{{\rm Case C:} \quad $\sgn_{A}(\tau_1(\alpha_{\sigma_1})) \neq \sgn_{A'}(\tau_1(\alpha_{\sigma_1}))  \textand  \tau_1 = \sigma_1$}
\end{align*}
We have
  \begin{align}
  \label{E:L-recursion}
L(\tau_1\leq\sigma_1, A) &=
    \begin{cases}
  L(\tau_1 \leq  \sigma_1, A') &\textif \mathrm{Case}\text{ }\mathrm{A} \\
  - L(\tau_1 \leq  \sigma_1, A')   &\textif \mathrm{Case}\text{ }\mathrm{B} \\
  L(\tau_1 \leq  \sigma_1, A')  +  L( e \leq  \sigma_1, A')  &\textif \mathrm{Case}\text{ }\mathrm{C}
 \end{cases}
\end{align}
and $\sD^{A'}_A(\tau\leq\sigma)$ is equal to
\begin{align}
 \label{E:D-recursion}
   &\begin{cases}
  \{ (\tau_1,\trho) \suchthat \trho \in \sD^{\tau_1^{-1}(A')}_{\tau_1^{-1}(A)}(\ttau \leq \tsigma) \}   &\textif \mathrm{Case}\text{ }\mathrm{A}\\
  \{ (\tau_1,\trho) \suchthat \trho \in \sD^{\tau_1^{-1}(A')}_{\tau_1^{-1}(A)}(\ttau \leq \tsigma) \}  &\textif \mathrm{Case}\text{ }\mathrm{B}\\
  \{ (\tau_1,\trho) \suchthat \trho \in \sD^{\tau_1^{-1}(A')}_{\tau_1^{-1}(A)}(\ttau \leq \tsigma) \} \,\bigsqcup\,\, \{ (e,\trho) \suchthat \trho \in \sD^{\tau_1^{-1}(A')}_{A}(\ttau \leq \tsigma) \}&\textif \mathrm{Case}\text{ }\mathrm{C}
%\mathperiod
 \end{cases} 
  \end{align}
where we write $\trho=(\rho_2,\dotsc,\rho_r)$.
\end{Proposition}

\begin{proof}[Proof of Theorem \ref{thm:change-of-basis-formula}]

Let us write $\sigma = (\sigma_1, \dotsc, \sigma_r)$ and $\tau = (\tau_1, \dotsc, \tau_r)$. We will proceed by induction on $r$.  %Let us write $\tsigma = (\sigma_2, \dotsc, \sigma_k)$, and let $\ttau = (\tau_2, \dotsc, \tau_k)$.

There are three cases according to the three cases in Proposition \ref{prop:L-recursion}. We will argue for the first case, i.e., let us assume $\sgn_{A}(\tau_1(\alpha_{\sigma_1})) = \sgn_{A'}(\tau_1(\alpha_{\sigma_1}))$. The other cases are similar and we omit them to avoid repetition.  

By \eqref{E:L-recursion} and Proposition~\ref{prop:simple-L-recursion} we have
\begin{align}
  \label{eq:10}
L(\tau\leq\sigma, A) = L(\tau_1 \leq  \sigma_1, A') L(\ttau\leq \tsigma, \tau_1^{-1}(A) ) 
  \mathperiod
\end{align}
By induction we have
\begin{align}
  \label{eq:11}
  &L(\ttau\leq \tsigma, \tau_1^{-1}(A) )\\ \notag&= \sum_{ \trho = (\rho_2, \dotsc, \rho_r) \in \sD^{\tau_1^{-1}(A')}_{\tau_1^{-1}(A)}(\ttau \leq \tsigma) } (-1)^{\#\cK_{\tau_1^{-1}(A)}^{\tau_1^{-1}(A')}(\trho\leq\ttau\leq\tsigma)} L(\trho \leq \ttau, \tau_1^{-1}(A') )
  \mathperiod
\end{align}
Thus \eqref{eq:10} equals
\begin{align}
\sum_{ \trho = (\rho_2, \dotsc, \rho_r) \in   \sD^{\tau_1^{-1}(A')}_{\tau_1^{-1}(A)}(\ttau \leq \tsigma) } (-1)^{\#\cK_{\tau_1^{-1}(A)}^{\tau_1^{-1}(A')}(\trho\leq\ttau\leq\tsigma)} L((\tau_1,\rho_2, \dotsc, \rho_r) \leq \tau, A' )
\end{align}
by Proposition \ref{prop:simple-L-recursion}. Finally, by \eqref{E:D-recursion}, this is equal to
\begin{align}
\sum_{ \rho = (\rho_1, \dotsc, \rho_r) \in   \sD^{A'}_{A}(\tau \leq \sigma)} (-1)^{\#\cK_{\tau_1^{-1}(A)}^{\tau_1^{-1}(A')}(\trho\leq\ttau\leq\tsigma)} L(\rho \leq \tau, A' )
  \mathperiod
\end{align}
Finally, by our assumption that $\sgn_{A}(\tau_1(\alpha_{\sigma_1})) = \sgn_{A'}(\tau_1(\alpha_{\sigma_1}))$, we have 
\begin{align}
{\#\cK_A^{A'}(\rho\leq\tau\leq\sigma)} =  {\#\cK_{\tau_1^{-1}(A)}^{\tau_1^{-1}(A')}(\trho\leq\ttau\leq\tsigma)} 
 \label{eq:12}
\end{align}
and the change of basis formula follows.
\end{proof}

\section{Counting points in Kac-Moody flag varieties}

Let $G$ be the Kac-Moody group associated with the root system of \S\ref{S:notation}, let $B\subset G$ be the positive Borel subgroup, and let $T\subset B$ be the maximal torus. (The reader may refer to \cite{Kum} for a thorough treatment of these notions.) We will be concerned with the Kac-Moody flag variety $G/B$, which is an ind-scheme. The fixed points for the $T$-action on $G/B$ are precisely indexed by $w \in W$. Let us abuse notation and write $w \in G/B$ for the corresponding fixed point and also $w\in G$ for fixed lifts of these points.

We call a cocharacter $\eta: \GG_m \rightarrow T$ regular if its pairing with any real root is nonzero. For any regular cocharacter $\eta$, we define an orientation $A_\eta$ by $\alpha >_{A_\eta} 0$ if any only if $\langle\eta, \alpha\rangle >0$. We will consider the change of basis between oriented bases corresponding to $A_0$ and $A_\eta$.
 
If $\sigma = (\sigma_1, \dotsc, \sigma_r)$ is a reduced expression, we have by Theorem~\ref{thm:change-of-basis-formula} that
\begin{align}
  \label{eq:39}
L(\sigma, A_0) = \sum_{\tau \in \sD^{\eta}(\sigma)} L(\tau \leq \sigma, A_\eta)
\end{align}
where, to simplify notation, we write $\sD^{\eta}(\sigma) = \sD^{A_\eta}_{A_0}(\sigma \leq \sigma)$.
Mapping to the Hecke algebra gives
\begin{align}
  \label{eq:001}
T(\sigma, A_0) = \sum_{\tau \in \sD^{\eta}(\sigma)} m^{\eta}_{\tau \leq \sigma}\,T(\pi(\tau), A_\eta)
\end{align}
where
\begin{align}
  m^{\eta}_{\tau \leq \sigma} = (v-v^{-1})^{\kappa_\tau}
  \mathperiod
\end{align}
(Note that we do not get any minus signs here because $\sigma$ is reduced and $A=A_0$.)
Let us write
\begin{align}
M^{\eta,u}_{\sigma} &= \sum_{\tau \in \sD^{\eta}(\sigma): \pi(\tau) = u} v^{ \kappa_\tau + 2 \zeta_\tau^+} m^{\eta}_{\tau \leq \sigma}
                      = \sum_{\tau \in \sD^{\eta}(\sigma): \pi(\tau) = u} v^{2\zeta_\tau^+}(v^2-1)^{\kappa_\tau}
                      \mathperiod
\end{align}

Now let us define $X^{\eta,u}_w \subset G/B$ by
\begin{align}
  \label{eq:40}
  X^{\eta,u}_w = \{ x \in B w \cdot B \suchthat  \lim_{z \rightarrow 0} \eta(z)\cdot x = u \}
  \mathperiod 
\end{align}
In the next section we will prove the following:
\begin{Theorem}
  \label{thm:point-count-in-flag-variety}
  Let $w\in W$, and let $\sigma$ be a reduced expression for $w$. Let $\eta$ be a regular cocharacter. Then for every $u\in W$, we have:
  \begin{align}
    \label{eq:41}
    \left.M^{\eta,u}_{\sigma}\right|_{v^2=q} = \# X^{\eta,u}_w (\FF_q)
    \mathperiod
  \end{align}
\end{Theorem}

\begin{Remark}\label{R:BD}
For any orientation $A$, let $Q_A\subset G$ be the subgroup generated by the one-parameter subgroups $e_\alpha : \GG_a\to G$ corresponding to $\alpha\in\Delta_+^A$. Then \cite[Theorem 1]{BD} asserts that the Bruhat-type decomposition $G=\sqcup_{w\in W} Q_A w B$ holds. In particular, it follows that for $A=A_\eta$ the limit $\lim_{z \rightarrow 0} \eta(z)\cdot x$ exists and is equal to an element of $W$ for any $x\in G/B$. Moreover, we see that
$X_w^{\eta,u}=Q_A u\cdot B\cap Bw\cdot B$.
We note that Theorem~\ref{thm:point-count-in-flag-variety} is a consequence of \cite[Theorem 5]{BD}, but we will give a different proof below.
\end{Remark}

\begin{Remark}
By Lemma \ref{L:kappa+2tau}, we have
\begin{align}
  \label{eq:32}
  M^{\eta,u}_{\sigma} = v^{\ell(w)+\ell_{A_\eta}(u)} \sum_{\tau \in \sD^{\eta}(\sigma): \pi(\tau) = u} m^{\eta}_{\tau \leq \sigma}
  \mathperiod 
\end{align}
In particular, up to an explicit constant depending only on $w$, $u$, and $\eta$, the point-count in Theorem~\ref{thm:point-count-in-flag-variety} is given exactly by summing coefficients $m^{\eta}_{\tau \leq \sigma}$.
\end{Remark}

\subsection{Bott-Samelson varieties}

For each $i \in I$, let us write $P_{i}\supset B$ for the corresponding rank-one parabolic subgroup; we will also write $P_{s_i} = P_{i}$. Let us also write $e_{ \pm \alpha_i} = e_{\pm \alpha_{s_i}} : \GG_a \rightarrow G$ for the corresponding one-parameter subgroups.

Fix an expression $\sigma=(\sigma_1,\dotsc,\sigma_r)$. Then we can form the Bott-Samelson variety:
\begin{align}
  \label{eq:14}
  \Sigma( \sigma ) = P_{\sigma_1} \times^{B} P_{\sigma_2} \times^{B} \cdots\times^B P_{\sigma_r} / B
\end{align}
which is a smooth projective variety.
Let us denote the points of $\Sigma(\sigma)$ by tuples $[g_1, \dotsc, g_r]$ where $g_j \in P_{\sigma_j}$ and we consider such tuples up to the following equivalence
\begin{align}
  \label{eq:16}
  [g_1, \dotsc, g_r] = [g_1 b_1, b_1^{-1} g_2 b_2, \dotsc, b_{r-1}^{-1} g_r b_{r}] 
\end{align}
for all $(b_1, \dotsc, b_r) \in B^{r}$.
There is a natural $B$-equivariant map
\begin{align}
  \label{eq:15}
  \pi : \Sigma( \sigma) \rightarrow G/B
\end{align}
given by 
\begin{align}
  \label{eq:17}
\pi( [g_1, \dotsc, g_r] ) = g_1 \cdots g_r B 
\end{align}
and the $B$-action on $\Sigma(\sigma)$ is given by left multiplication on the first component.
For any subexpression $\tau = (\tau_1, \dotsc, \tau_r)$, we can form the corresponding point $[\tau_1, \dotsc, \tau_r] \in \Sigma(\sigma)$. These points are precisely the $T$-fixed points of $\Sigma(\sigma)$.

For each subexpression $\tau \leq \sigma$ let us define
\begin{align}
  \label{eq:18}
  \cU_{\tau} = \{ [ \tau_1 e_{ - \alpha_{\sigma_1}}(z_1), \dotsc, \tau_r e_{ - \alpha_{\sigma_r}}(z_r)]  \suchthat z_i \in \GG_a \}
  \mathperiod
\end{align}
Then  $\cU_{\tau}$ is an open subset of $\Sigma(\sigma)$ isomorphic to the affine space $\AA^r$. The coordinates $(z_1,\dotsc, z_r)$ allow us to view $\cU_{\tau}$ as the tangent space to the point $\tau \in \Sigma(\sigma)$.

%Furthermore, these coordinates are compatible with the $T$-action.

Let $\eta : \GG_m \rightarrow T$ be a regular cocharacter. Then we can restrict the $T$-action on $\Sigma(\sigma)$ to a $\GG_m$-action via $\eta$. Because $\eta$ is regular, the fixed points of $\eta$ are exactly equal to the $T$-fixed points. Parallel to \eqref{eq:40}, we can therefore form the Bia\l{}ynicki-Birula cell
\begin{align}
  \label{eq:19}
  C^{\eta}_\tau = \{ x \in \Sigma(\sigma) \suchthat \lim_{z\rightarrow0} \eta(z)\cdot x = \tau \}
  \mathperiod
\end{align}
Because $\cU_{\tau}$ is a $T$-invariant open subset containing $\tau$, we must have
%\begin{align}
%  \label{eq:20}
$C^{\eta}_\tau \subset \cU_{\tau}.$
%\end{align}
Furthermore, the coordinates \eqref{eq:18} on $\cU_{\tau}$ allow us to explicitly compute the equations defining $C^{\eta}_\tau$. It is given by the vanishing of the coordinates $z_k$ that have negative weight under the action of $\eta(z)$. Under the action of $\eta(z)$, the coordinate $z_k$ is sent to $z^{ \langle \tau_k \cdots \tau_1(\eta), -\alpha_{\sigma_k} \rangle} z_k$. Therefore, we must set $z_k = 0$ if $\langle \eta, -\tau_1 \cdots \tau_k(\alpha_{\sigma_k}) \rangle < 0$. To summarize:
  \begin{align}
    \label{eq:22}
    C^{\eta}_\tau   = \{ [ \tau_1 e_{ - \alpha_{\sigma_1}}(z_1), \dotsc, \tau_r e_{ - \alpha_{\sigma_r}}(z_r)] \in \cU_\tau  \suchthat z_k = 0 \textif \langle \eta, -\tau_1 \cdots \tau_k(\alpha_{\sigma_k}) \rangle < 0 \}
    \mathperiod
  \end{align}

%We can also consider the orientation $A_{\eta}$ defined by $\beta >_{A_{\eta}} 0$ if $\langle \eta, \beta \rangle > 0$.  

\begin{Theorem}
  \label{thm:computing-intersection}
 Let $\sigma$ be a {\it reduced} expression. If $\tau \notin \sD^\eta(\sigma)$, then
 \begin{align}
   \label{eq:24}
   \pi^{-1}( B \pi(\sigma)\cdot B) \cap C^{\eta}_\tau = \emptyset
   \mathperiod
 \end{align}
If $\tau \in \sD^\eta(\sigma)$, then:
 \begin{align*}
   \pi^{-1}( B \pi(\sigma) \cdot B) \cap C^{\eta}_\tau
   & =\{ [ \tau_1 e_{ - \alpha_{\sigma_1}}(z_1), \cdots, \tau_r e_{ - \alpha_{\sigma_r}}(z_r)] \in \cU_\tau  \suchthat\\
   &\qquad\ z_k = 0 \textif \langle \eta, -\tau_1 \cdots \tau_k(\alpha_{\sigma_k}) \rangle < 0 \textand z_k \neq 0 \textif \tau_k = e\}
     \mathperiod 
 \end{align*} 
\end{Theorem}

In the proof of Theorem~\ref{thm:computing-intersection} we will use the following three standard facts.

\begin{Lemma}
  \label{lem:B-w-s-B}
  Suppose $w \in W$ and $s \in S$ with $\ell(ws) = \ell(w) + 1$. Then, we have 
\begin{align} \label{eq:25}
  B w B s B = B ws B
  \mathperiod
  \end{align}
\end{Lemma}

\begin{Lemma}
  \label{lem:B-s-B}
Let $s \in S$. Then we have
\begin{align}
  \label{eq:26}
  B  e_{-\alpha_{s}}(z) B \subseteq B s B
  \mathperiod 
\end{align}
for $z \in \GG_m \subset \GG_a$.
\end{Lemma}

% \begin{Lemma}
% Suppose $\sigma = (\sigma_1, \cdots, \sigma_r)$ is a {\it reduced} expression. Then we have
% \begin{align}
%   \label{eq:26}
%   B \sigma_1 \cdots \sigma_{r-1} e_{-\alpha_{\sigma_r}}(z) B \subseteq B \sigma_1 \cdots \sigma_r  B 
% \end{align}
% for $z \in \GG_m \subset \GG_a$.
% \end{Lemma}

\begin{Lemma}
  \label{lem:inclusion-of-sigma-k}
  Let $\sigma^{(k)} = (\sigma_1, \dotsc, \widehat{\sigma_k}, \dotsc, \sigma_r)$. Then we have a closed embedding 
  \begin{align}
    \label{eq:21}
    i^{(k)}: \Sigma( \sigma^{(k)} ) \hookrightarrow \Sigma(\sigma) 
  \end{align}
  given by
  \begin{align}
    \label{eq:23}
    [g_1,\dotsc, g_{k-1}, g_{k+1}, \dotsc, g_r] \mapsto [g_1,\dotsc, g_{k-1}, e,  g_{k+1}, \dotsc, g_r]
    \mathperiod
  \end{align}
  Furthermore, the map $i^{(k)}$ intertwines the two maps $\pi : \Sigma( \sigma^{(k)}) \rightarrow G/B$ and $\pi : \Sigma( \sigma) \rightarrow G/B$.
\end{Lemma}

The following Lemma is a simple application of the Borel fixed point theorem (see \cite[Corollary 2.2.2]{Brion} for a similar argument).
\begin{Lemma}
Suppose $v \in W$ is in the image of $\pi : \Sigma( \sigma) \rightarrow G/B$. Then there exists $\tau \leq \sigma$ such that $\pi(\tau) = v$.  
\end{Lemma}

From this lemma we can deduce the following.
\begin{Lemma}
  \label{lem:image-of-sigma-sigma-k}
 Suppose  $\sigma = (\sigma_1, \dotsc, \sigma_r)$ is a reduced expression. Let $\sigma^{(k)} = (\sigma_1, \dotsc, \widehat{\sigma_k}, \dotsc, \sigma_r)$. Then $\pi(\Sigma(\sigma^{(k)})) \cap B \pi(\sigma)\cdot B = \emptyset$ for all $k \in [r]$.
\end{Lemma}

\newcommand{\otau}{\overline{\tau}}
\newcommand{\osigma}{\overline{\sigma}}

\begin{proof}[Proof of Theorem \ref{thm:computing-intersection}]
 
Define 
\begin{align}
  %\label{eq:31}
\notag
S^\eta_\tau = 
& \{ [ \tau_1 e_{ - \alpha_{\sigma_1}}(z_1), \dotsc, \tau_r e_{ - \alpha_{\sigma_r}}(z_r)] \in \cU_\tau  \suchthat\\
\notag
&\qquad\qquad z_k = 0 \textif \langle \eta, -\tau_1 \cdots \tau_k(\alpha_{\sigma_k}) \rangle < 0 \textand z_k \neq 0 \textif \tau_k = e\} 
\end{align}
and consider a point $[ \tau_1 e_{ - \alpha_{\sigma_1}}(z_1), \dotsc, \tau_r e_{ - \alpha_{\sigma_r}}(z_r)] \in C^\eta_\tau \backslash S^{\eta}_\tau$. We must have $\tau_k = e$ and $z_k =0$ for some $k$. Hence, the point $[ \tau_1 e_{ - \alpha_{\sigma_1}}(z_1), \dotsc, \tau_r e_{ - \alpha_{\sigma_r}}(z_r)]$ belongs to $i^{(k)}( \Sigma( (\sigma_1, \dotsc, \widehat{\sigma_k}, \dotsc, \sigma_r) )$. 
Therefore, $\pi([ \tau_1 e_{ - \alpha_{\sigma_1}}(z_1), \dotsc, \tau_r e_{ - \alpha_{\sigma_r}}(z_r)])$ does not belong to $B \pi(\sigma)\cdot B$ by Lemmas \ref{lem:inclusion-of-sigma-k} and \ref{lem:image-of-sigma-sigma-k}. We conclude that
 \begin{align*}
   & \pi^{-1}( B \pi(\sigma) \cdot B) \cap C^{\eta}_\tau \subseteq  S^{\eta}_\tau
     \mathperiod
 \end{align*} 

Conversely suppose we have
%\begin{align}
%  \label{eq:27}
$[ \tau_1 e_{ - \alpha_{\sigma_1}}(z_1), \dotsc, \tau_r e_{ - \alpha_{\sigma_r}}(z_r)] \in S^{\eta}_{\tau}$.
%\end{align}
We need to show
\begin{align}
  \label{eq:28}
    \tau_1 e_{ - \alpha_{\sigma_1}}(z_1) \cdots \tau_r e_{ - \alpha_{\sigma_r}}(z_r) \in B \pi(\sigma) B
  \mathperiod
\end{align}
We proceed by induction on $r$. By induction, we have
\begin{align}
  \label{eq:29}
  \tau_1 e_{ - \alpha_{\sigma_1}}(z_1) \cdots \tau_{r-1} e_{ - \alpha_{\sigma_{r-1}}}(z_{r-1}) \in B \sigma_1 \cdots \sigma_{r-1}B
  \mathperiod
\end{align}
We also have
\begin{align}
  \label{eq:30}
  \tau_r e_{ - \alpha_{\sigma_r}}(z_r)  \in B \sigma_r B
  \mathperiod  
\end{align}
Indeed, if $\tau_r = e$, then $z_r \in \GG_m$ and
\eqref{eq:30} holds by Lemma \ref{lem:B-s-B}. If $\tau_r = \sigma_r$, then we have $\sigma_r e_{ - \alpha_{\sigma_r}}(z_r) =  e_{ \alpha_{\sigma_r}}(z_r) \sigma_r$, hence we have \eqref{eq:30}. Then by \eqref{eq:29} and Lemma \ref{lem:B-w-s-B}  we have \eqref{eq:28}. Therefore:
\begin{align}
   \notag	
  & \pi^{-1}( B \pi(\sigma)\cdot B) \cap C^{\eta}_\tau = S^{\eta}_\tau
    \mathperiod
 \end{align} 

Finally, if $\tau \notin \sD^{A_{\eta}}_{A_0}(\sigma)$,
then we must have $\tau_k = e$ and $\langle \eta, -\tau_1 \cdots \tau_k(\alpha_{\sigma_k}) \rangle < 0 $ for some $k$. Then we must have $z_k = 0$ by \eqref{eq:22}, which contradicts the conditions on $[ \tau_1 e_{ - \alpha_{\sigma_1}}(z_1), \dotsc, \tau_r e_{ - \alpha_{\sigma_r}}(z_r)]$ in the definition of $S^\eta_\tau$. Hence $S^\eta_\tau = \emptyset$ in this case.
\end{proof}

\begin{proof}[Proof of Theorem \ref{thm:point-count-in-flag-variety}]

Let $w \in W$, and let $\sigma$ be a reduced expression for $w$. Because the map $\pi: \Sigma(\sigma) \rightarrow G/B$ is a bijection over $B w\cdot B \subset G/B$, it induces a bijection of sets:
\begin{align}
  \label{eq:42}
  \bigsqcup_{\tau \in \sD^{\eta}(\sigma): \pi(\tau) = u}\pi^{-1}(B w\cdot B) \cap C^\eta_\tau \cong X^{\eta,u}_w
  \mathperiod
\end{align}
By Theorem \ref{thm:computing-intersection} we conclude that
\begin{align}
  \label{eq:43}
 \# X^{\eta,u}_w(\FF_q) &= \#_{\FF_q}\left(\bigsqcup_{\tau \in \sD^{\eta}(\sigma): \pi(\tau) = u}\pi^{-1}(B w\cdot B) \cap C^\eta_\tau   \right)\\
  \notag
  &= 
\sum_{\tau \in \sD^{\eta}(\sigma): \pi(\tau) = u} v^{2\zeta_\tau^+}(v^2-1)^{\kappa_\tau}|_{v^2=q}
\end{align}
where $\zeta_\tau^+$ and $\kappa_\tau$ are defined by \eqref{E:zeta+} and \eqref{E:kappa}. 
\end{proof}

\section{Some connections to other work}

\subsection{Deodhar's formula for $R$-polynomials}
The Kazhdan-Lusztig involution \cite{KL} is the $\ZZ$-linear algebra involution $\overline{\,\cdot\,}: \cH \rightarrow \cH$ on the Hecke algebra given on generators by $\overline{T}_i = T_i^{-1}$ for all $i \in I$ and $\overline{v}=v^{-1}$. The $R$-polynomials are essentially the matrix of the Kazhdan-Lusztig involution. More precisely, for any pair $u,w \in W$, there is a polynomial $R_{u,w} \in \ZZ[v,v^{-1}]$ defined by
\begin{align}
  \label{eq:37}
  T_w  = \sum_{u \in W} v^{\ell(u) - \ell(w)}R_{u,w} \overline{T}_u
  \mathperiod
\end{align}
It is clear that $ \overline{T}_u = T(u,-A_0)$, so applying \eqref{eq:001}, we have the following formula for the $R_{u,w}$.
\begin{Theorem}\cite[Theorem 1.3]{Deo}
  Let $\sigma$ be a reduced expression for $w$. Then we have
  \begin{align}
    \label{eq:38}
    R_{u,w} = \sum_{\tau \in \sD_{A_0}^{-A_0}(\sigma): \pi(\tau)=u} v^{\ell(w) - \ell(u)}\,m_{\tau\leq\sigma}^{-A_0}
    \mathperiod
  \end{align}
\end{Theorem}

Theorem~\ref{thm:point-count-in-flag-variety} translates in this special case to the fact that $R_{u,w}|_{v^2=q}$ is equal to $\#_{\FF_q}(B^- u\cdot B\cap B w\cdot B)$, where $B^-$ is the Borel subgroup opposite to $B$.

\subsection{Affine type}
For the rest of the paper, let us consider the case when $W$ is the Weyl group of an affine Kac-Moody group $G$. For simplicity we assume that $G$ is of untwisted affine type. Thus $W=Q^\vee_\circ\rtimes W_\circ$, where $W_\circ$ and $Q^\vee_\circ$ are the Weyl group and coroot lattice of a finite type Kac-Moody group $G_\circ$. Write $I=I_\circ\sqcup\{0\}$, where $I_\circ$ is the index set for the simple roots of $G_\circ$.

The Hecke algebra $\cH$ of $W$, the affine Hecke algebra, has another presentation known as the Bernstein presentation, which is analogous to the semidirect product decomposition of $W$. This is expressed as a linear isomorphism
$
\cH \cong \ZZ[v^{\pm 1}][Q^\vee_\circ]\otimes_{\ZZ[v^{\pm 1}]}\cH_\circ
$
where $\cH_\circ$ is the Hecke algebra of $W_\circ$ and $\ZZ[v^{\pm 1}][Q^\vee_\circ]$ is the group algebra of $Q^\vee_\circ$, with basis elements $X^\lambda$ for $\lambda\in Q^\vee_\circ$. These elements satisfy $X^\lambda X^\mu = X^{\lambda+\mu}$, and the tensor factors in the isomorphism above each form subalgebras of $\cH$. The relations between these subalgebras are given by
\begin{align}\label{E:TX-reln}
T_iX^\lambda = X^{s_i(\lambda)}T_i + (v-v^{-1})\frac{X^\lambda-X^{s_i(\lambda)}}{1-X^{\alpha_i}}
\end{align}
for $i\in I_\circ$ and $\lambda\in Q^\vee_\circ$. For more details on this presentation of $\cH$, we refer the reader to \cite{Lus}.

Let $\theta$ be the highest (long) root of $G_\circ$. The element $T_0$ from $\cH$ is given in the Bernstein presentation as $T_0 = T_{s_\theta}^{-1} X^{-\theta^\vee}$.

We have a basis of $\cH$ given by $\{X^\lambda T_u \suchthat \lambda\in Q^\vee_\circ,\, u\in W_\circ\}$ in the Bernstein presentation. It turns out that this is the oriented basis for the periodic orientation. More precisely, for any element $x=\varpi^\lambda u\in W$, one has $T(x,A_\infty)=X^{\wt(x)}T_{\dir(x)}=X^\lambda T_u$ (see, e.g., \cite[(3.2.10) and (3.5.1)]{Mac}).

Let $\bI\subset G$ be the Iwahori subgroup (i.e., the standard Borel subgroup of $G$) and let $\bU^\pm$ be the subgroups of $G$ corresponding to $A_{\pm\infty}$ as in Remark~\ref{R:BD}. 
By Theorem~\ref{thm:point-count-in-flag-variety}, the coefficient of the Bernstein basis element $T(x,A_\infty)$ in $T(w,A_0)$ for $x,w\in W$ in the expansion \eqref{eq:001} is proportional to 
$\#_{\FF_q}\left(\bU^+x\cdot \bI \cap \bI w\cdot \bI\right)$, with the coefficient of proportionality given by $v^{\ell(w)+\ell_{A_\infty}(x)}=v^{\ell(w)+\ell(u)-\langle\lambda,2\rho\rangle}$ if $x=\varpi^\lambda u$.
A similar interpretation holds for $T(x,A_{-\infty})$, with the $\bU^-$-orbit of $x$ instead of the $\bU^+$-orbit; in this special case Theorem~\ref{thm:point-count-in-flag-variety} was proved in \cite[Theorem 7.1]{PRS}.

\begin{Remark}
 The groups $\bU^\pm$ also have the following concrete realization. As $G$ is untwisted affine, it can be realized as a central extension of a loop group of $G_\circ$. For us the relevant point is that there is a map $G \rightarrow G_\circ$ given by evaluating a loop at $1$. Let $U^\pm_\circ$ be the positive and negative unipotent subgroups of $G_\circ$ corresponding to the set of positive and negative roots. Modulo the center of $G$, $\bU^\pm$ is precisely the preimage of $U^\pm_\circ$ under the map $G \rightarrow G_\circ$. Because the central extension splits over $U^\pm_\circ$, we can use the splitting to explicitly see that $\bU^\pm$ is a copy of the loop group of $U^\pm_\circ$ embedded as a subgroup of $G$.
\end{Remark}

\subsection{Coefficients of nonsymmetric Macdonald polynomials}
Continuing with the notation from the previous subsection, let us consider the polynomial representation of the affine Hecke algebra $\cH$. This can be described abstractly as the induced representation $\mathrm{Ind}_{\cH_\circ}^\cH (\mathrm{triv})$, where $\mathrm{triv}$ is the one-dimensional representation in which $T_i$ for $i\in I_\circ$ acts by the scalar $v$. Using the Bernstein presentation and in particular \eqref{E:TX-reln}, one sees that the polynomial representation is linearly isomorphic to $\ZZ[v^{\pm 1}][Q^\vee_\circ]$, with the action of $X^\lambda$ given by multiplication in the group algebra and the action of $T_i$ for $i\in I_\circ$ given by the Demazure-Lusztig operator
\begin{align}
  T_i = v s_i + \frac{v-v^{-1}}{1-X^{\alpha_i}}(1-s_i)
  \mathperiod
\end{align}
The action of the element $T_0$ from $\cH$ in the polynomial representation is given via the Bernstein presentation by the formula $T_0 = T_{s_\theta}^{-1} X^{-\theta^\vee}$.

The nonsymmetric Macdonald polynomials $E_\lambda(\mathsf{q},v)$ for $\lambda\in Q^\vee_\circ$ form a distinguished basis in the polynomial representation of $\cH$ (with coefficients extended to $\QQ(\mathsf{q},v)$ --- we use the letter $\mathsf{q}$ to distinguish this parameter from the $q$ in $\FF_q$). We refer the interested reader to \cite[\S3.3]{Ch2} for an introduction to these remarkable polynomials and a discussion of their origins. Our present aim is to relate the specializations of $E_\lambda(\mathsf{q},v)$ at $\mathsf{q}=0,\infty$ to the change of basis and point-counting problems considered in this paper.

To this end, we make use of Cherednik's ``intertwiner'' construction of the $E_\lambda(\mathsf{q},v)$, which is based on the operators in the polynomial representation (see \cite[Theorem 5.1]{Ch1}, or \cite[(3.3)]{RY} for our conventions). Fix a reduced expression $\sigma_\lambda=(s_{i_1},\dotsc,s_{i_r})$ such that $\pi(\sigma_\lambda)=m_\lambda$, where $m_\lambda=\varpi^\lambda u_\lambda^{-1}$ is the minimum coset representative for $\varpi^\lambda$ in $W/W_\circ$. Then the intertwiner construction is as follows:
\begin{align}\label{E:intertwiners}
&E_\lambda(\mathsf{q},v)\\ \notag&= v^{-\ell(u_\lambda)}\left(T_{i_1}+\frac{v-v^{-1}}{\mathsf{q}^{a_1}v^{b_1}-1}\right)\left(T_{i_2}+\frac{v-v^{-1}}{\mathsf{q}^{a_2}v^{b_2}-1}\right)\dotsm\left(T_{i_r}+\frac{v-v^{-1}}{\mathsf{q}^{a_r}v^{b_r}-1}\right)\cdot 1
\end{align}
where $a_k\in\ZZ_{>0}$ and $b_k\in\ZZ$ are explicit integers. 

\comment{
({\tt\color{red} *** Dinakar (to remove later):} The monomial $q^{a_k}v^{b_k}$ is the eigenvalue of $Y^{-\alpha_{i_k}}$ on $E_\mu$ where $\mu+d={s_{i_{k+1}}\dotsm s_{i_\ell}(\lambda+d)\mod c}$, i.e., this is the level one action. Explicitly, $a_k=(\alpha_i,\mu+d)$ and $b_k=-(\alpha_i,u_\mu^{-1}(\rho^\vee))$ where $u_\mu$ is the shortest finite Weyl group element such that $u_\mu(\mu)$ is antidominant. Since we work with a min rep $m_\lambda$ we know that $a_k>0$. Actually, we are guaranteed $b_k\neq 0$, and $b_k>0$ whenever $i_k\neq 0$ as well. {\tt\color{red}***})
}

At the specializations $\mathsf{q}=0,\infty$ we obtain
\begin{align}
E_\lambda(\infty,v) &= v^{-\ell(u_\lambda)}T_{i_1}T_{i_2}\dotsm T_{i_r}\cdot 1\notag\\
&= v^{-\ell(u_\lambda)}T(\sigma_\lambda,A_0)\cdot 1\\
E_\lambda(0,v) &= v^{-\ell(u_\lambda)}\left(T_{i_1}-(v-v^{-1})\right)\left(T_{i_2}-(v-v^{-1})\right)\dotsm\left(T_{i_r}-(v-v^{-1})\right)\cdot 1\notag\\
&= v^{-\ell(u_\lambda)}T_{i_1}^{-1}T_{i_2}^{-1}\dotsm T_{i_r}^{-1}\cdot 1\notag\\
                    &= v^{-\ell(u_\lambda)}T(\sigma_\lambda,-A_0)\cdot 1
                      \mathperiod
\end{align}

In order to compute more effectively in the polynomial representation, we consider the change of basis to the Bernstein basis (the oriented basis corresponding to $A_\infty$):
\begin{align}
T(\sigma_\lambda,A_0) &= \sum_{\tau\in\sD_{A_0}^{A_\infty}(\sigma_\lambda)} m_{\tau\le\sigma}^{A_\infty}\,T(\pi(\tau),A_{\infty})\\
T(\sigma_\lambda,-A_0) &= \overline{T(\sigma_\lambda,A_0)}\\
\notag
&= \sum_{\tau\in\sD_{A_0}^{-A_\infty}(\sigma_\lambda)}\overline{m^{-A_\infty}_{\tau \leq \sigma}\,T(\pi(\tau),-A_\infty)}\\
\notag
                      &= \sum_{\tau\in\sD_{A_0}^{-A_\infty}(\sigma_\lambda)}\overline{m^{-A_\infty}_{\tau \leq \sigma}}\,T(\pi(\tau),A_\infty)
                        \mathperiod
\end{align}
Applying these to $1$ in the polynomial representation, we obtain
\begin{align}
\label{E:Eatinftysub}
E_\lambda(\infty,v) &= v^{-\ell(u_\lambda)}T(\sigma_\lambda,A_0)\cdot 1\\
\notag
&= v^{-\ell(u_\lambda)}\sum_{\tau\in\sD_{A_0}^{A_\infty}(\sigma_\lambda)} m_{\tau\le\sigma}^{A_\infty}\,T(\pi(\tau),A_{\infty})\cdot 1\\
\notag
&=v^{-\ell(u_\lambda)}\sum_{\tau\in\sD_{A_0}^{A_\infty}(\sigma_\lambda)} m_{\tau\le\sigma}^{A_\infty}\,X^{\wt(\pi(\tau))}v^{\ell(\dir(\pi(\tau)))}\\
\label{E:Eat0sub}
E_\lambda(0,v) &= v^{-\ell(u_\lambda)}T(\sigma_\lambda,-A_0)\cdot 1\\
\notag
&= v^{-\ell(u_\lambda)}\sum_{\tau\in\sD_{A_0}^{-A_\infty}(\sigma_\lambda)}\overline{m^{-A_\infty}_{\tau \leq \sigma}}\,T(\pi(\tau),A_\infty)\cdot 1\\
\notag
                    &= v^{-\ell(u_\lambda)}\sum_{\tau\in\sD_{A_0}^{-A_\infty}(\sigma_\lambda)}\overline{m^{-A_\infty}_{\tau \leq \sigma}}X^{\wt(\pi(\tau))}v^{\ell(\dir(\pi(\tau)))}
                      \mathperiod
\end{align}
Let $[X^\nu]f$ denote the coefficient of the monomial $X^\nu$ in $f\in\ZZ[v^{\pm 1}][Q_\circ^\vee]$. Using Theorem~\ref{thm:point-count-in-flag-variety} and Lemma~\ref{L:A-inf-length}, we deduce the following:
\begin{Theorem}
For any $\lambda,\nu\in Q^\vee_\circ$, we have
\begin{align}
\label{E:Eatinfty}
\left.[X^\nu]E_\lambda(\infty,v)\right|_{v^2=q}&= q^{\langle\lambda_-+\nu,\rho\rangle}\#_{\FF_q}\left(\sqcup_{u\in W_\circ}\bU^+\varpi^\nu u\cdot \bI\cap \bI m_\lambda\cdot \bI\right)\\
\label{E:Eat0}
\left.[X^\nu]E_\lambda(0,v)\right|_{v^{-2}=q}&=
q^{\langle\lambda_--\nu,\rho\rangle+\ell(u_\lambda)}\#_{\FF_{q}}\left(\sqcup_{u\in W_\circ}\bU^-\varpi^\nu u\cdot \bI\cap \bI m_\lambda\cdot \bI\right)
\end{align}
where $\lambda_-$ is the unique antidominant weight in the $W_\circ$-orbit of $\lambda$.
\end{Theorem}

\comment{
By \cite[Proposition 6.1]{OS}, we have
\begin{align}
[X^\mu]E_\lambda(0,t)
&= \sum_{\substack{p\in\mathcal{B}^+(m_\lambda)\\\wt(p)=\mu}}v^{\ell(\dir(p))+|J(p)|-\ell(u_\lambda^{-1})}(t^{-1}-1)^{|J(p)|}\\
&= \sum_{\substack{p\in\mathcal{B}^+(m_\lambda)\\\wt(p)=\mu}}v^{\ell(\dir(p))+2m_+(p)+n(p)-\ell(u_\lambda^{-1})}t^{-m_+(p)}(t^{-1}-1)^{n(p)}.
%&= \sum_{\substack{p\in\mathcal{B}^+(m_\lambda)\\\wt(p)=\mu}}v^{\ell(\dir(p))+m_+(p)+n(p)+m_-(p)+m_+(p)-m_-(p)-\ell(u_\lambda^{-1})}t^{-m_+(p)}(t^{-1}-1)^{n(p)}
\end{align}
}

\comment{
\begin{align}
\ell(\dir(p))+2m_+(p)+n(p)-\ell(u_\lambda^{-1}) &= \ell(\dir(p))+\ell(m_\lambda)+\ell_{A_{-\infty}}(\ep(p))-\ell(u_\lambda^{-1})\\
&= \ell(\dir(p))+\ell(m_\lambda)+\langle\mu,2\rho\rangle-\ell(\dir(p))-\ell(u_\lambda^{-1})\\
&= \ell(m_\lambda)+\langle\mu,2\rho\rangle-\ell(u_\lambda^{-1})\\
&= \ell(t_\lambda)+\langle\mu,2\rho\rangle-2\ell(u_\lambda^{-1})\\
&= -\langle\lambda_--\mu,2\rho\rangle-2\ell(u_\lambda^{-1}).
\end{align}
}

\comment{
Hence
\begin{align}
[X^\mu]E_\lambda(0,t) %&= t^{-\langle\lambda_--\mu,\rho\rangle-\ell(u_\lambda)}\sum_{\substack{\tau\in\sD_{A_0}^{A_\infty}(m_\lambda)\\\wt(\tau)=\mu}} t^{-m_+(\tau)}(t^{-1}-1)^{n(\tau)}\\
&= t^{-\langle\lambda_--\mu,\rho\rangle-\ell(u_\lambda)}|(\sqcup_{u\in W_\circ}U^- \varpi^\mu u\cdot I)\cap I m_\lambda \cdot I|_{\mathbb{F}_{t^{-1}}}.
%&= t^{-\langle\lambda_--\mu,\rho^\vee\rangle-\ell(u_\lambda)}\sum_{w\in W_\circ}t^{\ell(w)}|\mathbf{U}^- \varpi^\mu\cdot I\cap I m_\lambda w^{-1} \cdot I|_{\mathbb{F}_{t^{-1}}}\\
%&= t^{-\langle\lambda_--\mu,\rho^\vee\rangle-\ell(u_\lambda)}\sum_{w\in W_\circ}t^{\ell(w)}t^{-\ell_{A_\infty}(\varpi^\mu)}|\mathbf{U}^- \varpi^{-\mu}\cdot I\cap I wm_\lambda^{-1} \cdot I|_{\mathbb{F}_{t^{-1}}}\\
%&= t^{-\langle\lambda_-+\mu,\rho^\vee\rangle-\ell(u_\lambda)}|\bU^- \varpi^\mu \cdot \bG_\circ(\cO)\cap I \varpi^\lambda\cdot \bG_\circ(\cO)|_{\mathbb{F}_{t^{-1}}}\\
%&= t^{-\langle\lambda_-+\mu,\rho^\vee\rangle-\ell(u_\lambda)+\ell(w_0)}|\bU^- \varpi^\mu \cdot I\cap I m_\lambda\cdot I|_{\mathbb{F}_{t^{-1}}}
%&= t^{-\langle\lambda_-+\mu,\rho^\vee\rangle-\ell(u_\lambda)}\mathrm{vol}(\bU^- t^\mu \bG_\circ(\cO)\cap I t^\lambda \bG_\circ(\cO))/\mathrm{vol}(\bG_\circ(\cO))
\end{align}
}

\begin{Remark}
Formulas \eqref{E:Eatinftysub} and \eqref{E:Eat0sub} are special cases of the Ram-Yip formula \cite[Theorem~3.1]{RY}, which gives a combinatorial expression for $E_\lambda(\mathsf{q},v)$ with both parameters. Our change of basis computations leading to these formulas are entirely parallel to the proof of the Ram-Yip formula, but they are simpler at $\mathsf{q}=0,\infty$. 
\end{Remark}

\subsubsection{Comparison with Ion's formula}
Ion has proved a slightly different version of \eqref{E:Eat0} in \cite[Theorem 4.2]{Ion}, namely
\begin{align}\label{E:Ion-formula}
\left.[X^\nu]E_\lambda(0,v)\right|_{v^{-2}=q} &= q^{\langle\lambda_-+\nu,\rho\rangle+\ell(u_\lambda)-\ell(w_0)}\#_{\FF_q}(\bU^- \varpi^{-\nu} \cdot \bI\cap \bK\varpi^{-\lambda} \cdot \bI)
\end{align}
where $\bK=\sqcup_{u\in W_\circ}\bI w\bI$ and $w_0$ is the long element of $W_\circ$.
One can establish the equality of the two formulas as follows.\footnote{It may help the reader to be aware that \cite{Ion} uses different conventions for the Macdonald polynomial parameters---namely our $\mathsf{q}=0$ and $v$ arbitrary corresponds to his $\mathsf{q}=\infty$ and $t=v^{-2}$.}

Let $G\xrightarrow{r} G/\bI \xrightarrow{p} G/\bK$ be the canonical projections. Observe that $p(X\cap Y)=p(X)\cap p(Y)$ for any $X,Y\subset G/\bI$ such that $r^{-1}(X)$ is right $\bK$-stable. Both of the sets $\sqcup_{u\in W_\circ}\bU^-\varpi^\nu u\bI$ and $\sqcup_{u\in W_\circ} \bI m_\lambda u\bI$ are right $\bK$-stable (for the former, see, e.g., \cite[Corollary 1]{BD}).

Since $m_\lambda$ is a minimal coset representative, the restriction $p_{|_{\bI m_\lambda\cdot \bI}} : \bI m_\lambda \cdot \bI \to \bI \varpi^\lambda \cdot \bK$ is an isomorphism. Hence it restricts further to a bijection between $(\sqcup_{u\in W_\circ}\bU^-\varpi^\nu u\cdot \bI) \cap \bI m_\lambda \cdot \bI$ and its image under $p$. We obtain
\begin{align}
  \#_{\FF_q}(\sqcup_{u\in W_\circ}\bU^-\varpi^\nu u\cdot \bI \cap \bI m_\lambda \cdot \bI) &= \#_{\FF_q}(\bU^-\varpi^\nu\cdot \bK\cap \bI \varpi^\lambda\cdot \bK)\label{E:1}
\mathperiod                                                                                             
\end{align}

On the other hand, one can verify each fiber of the map $p_{|_{\bU^-\varpi^\nu\cdot \bI}} : \bU^-\varpi^\nu\cdot \bI \to \bU^-\varpi^\nu\cdot \bK$ contains exactly $q^{\ell(w_0)}$ points, and that
%because each fiber can be identified with $U^-(\cO)/U^-_\varpi$ (up to conjugation). 
the same is true for the further restriction of $p$ to $\bU^-\varpi^\nu\cdot \bI \cap (\sqcup_{u\in W_\circ} \bI m_\lambda u\cdot \bI)$. Therefore
\begin{align}
\#_{\FF_q}(\bU^-\varpi^\nu\cdot \bK\cap \bI \varpi^\lambda\cdot \bK) =
  q^{-\ell(w_0)}\#_{\FF_q}(\bU^-\varpi^\nu\cdot \bI \cap (\sqcup_{u\in W_\circ} \bI m_\lambda u\cdot \bI))\label{E:2}
  \mathperiod
\end{align}

Finally, we claim that for any $x\in W$ and $\nu\in Q_\circ^\vee$ one has 
\begin{align}
  \#_{\FF_q}(\bU^-\varpi^\nu\cdot\bI \cap \bI x\cdot\bI) &= q^{\langle\nu ,2\rho\rangle}\#_{\FF_q}(\bU^-\varpi^{-\nu}\bI\cap\bI x^{-1}\cdot \bI)
\mathperiod                                                          
\end{align}
This follows from Theorem~\ref{thm:point-count-in-flag-variety} and the following lemma, which gives a termwise equality of the respective sums from \eqref{eq:41}, up to the factor $\ell_{A_{-\infty}}(\varpi^\nu)=\langle\nu,2\rho\rangle$.
\begin{Lemma}
Suppose $\nu\in Q^\vee_\circ$ and $\sigma$ is a reduced expression of length $r$. Then the map $\tau=(\tau_1,\dotsc,\tau_r)\mapsto(\tau_r,\dotsc,\tau_1)=:\tau'$ reversing a subexpression gives a bijection
\begin{align}
\{\tau\in\sD_{A_0}^{A_{-\infty}}(\sigma) \suchthat \pi(\tau)=\varpi^\nu\}&\xrightarrow{\sim}\{\tau\in\sD_{A_0}^{A_{-\infty}}(\sigma')\suchthat \pi(\tau)=\varpi^{-\nu}\}
\end{align}
such that $\kappa_\tau=\kappa_{\tau'}$ and $\zeta^\pm_\tau=\zeta^\mp_{\tau'}$.
\end{Lemma}

\comment{Notice that if $\pi(\tau)=\varpi^\nu$ then we have $\varpi^\nu\tau_r\dotsm \tau_k(\alpha_{\sigma_k}) = \pm\tau_1\dotsm\tau_k(\alpha_{\sigma_k})$ where $\pm=+$ if $\tau_k=e$ and $\pm=-$ if $
\tau_k=\sigma_k$. Since the action of $\varpi^\nu$ does change the finite part of an affine root, we have $\sgn_{A_{-\infty}}(\tau_r\dotsm\tau_k(\alpha_{\sigma_k}))=\pm\sgn_{A_{-\infty}}(\tau_1\dotsm\tau_k(\alpha_{\sigma_k}))$ with the sign as above. So the bijection preserves $\sD_{A_0}^{A_{-\infty}}(\sigma)$, preserves $\kappa_\tau$, and exchanges $\zeta_\tau^+$ with $\zeta^-_\tau$. The difference between the latter is $\ell_{A_{-\infty}}(\varpi^\nu)=\langle\nu,2\rho\rangle$.}

Therefore
\begin{align}
\#_{\FF_q}(\bU^-\varpi^\nu\cdot \bI \cap (\sqcup_{u\in W_\circ} \bI m_\lambda u\cdot \bI)) &= q^{\langle \nu, 2\rho\rangle}\#_{\FF_q}(\bU^-\varpi^{-\nu}\cdot \bI \cap (\sqcup_{u\in W_\circ} \bI u^{-1} m_\lambda^{-1}\cdot \bI))\notag\\
                                                                                           &= q^{\langle \nu, 2\rho\rangle}\#_{\FF_q}(\bU^-\varpi^{-\nu}\cdot \bI \cap \bK\varpi^{-\lambda}\cdot \bI)\label{E:3}
\mathperiod                                                                                            
\end{align}
By stringing together \eqref{E:1},\eqref{E:2}, and \eqref{E:3}, our formula \eqref{E:Eat0} is converted to Ion's formula \eqref{E:Ion-formula}.


\begin{thebibliography}{[BKP]}

\bibitem[BGR]{BGR}
N. Bardy-Panse, S. Gaussent, and G. Rousseau.
``Iwahori-Hecke algebras for Kac-Moody groups over local fields.''
{\it Pacific J. Math.},
{\bf 285},
{(2016)},
{no. 1},
{1--61}.
		


\bibitem[BD]{BD}
Y. Billig, and M. Dyer.
``Decompositions of Bruhat type for the Kac-Moody groups.''
{\it Nova J. Algebra Geom.} {\bf 3} (1994), no. 1, 11--39.

\bibitem[BKP]{BKP}
A. Braverman, D. Kazhdan, and M. Patnaik.
``Iwahori-Hecke algebras for $p$-adic loop groups.''
{\it Invent. Math.} {\bf 204} (2016), no. 2, 347--442. 

\bibitem[Bri]{Brion}
{M. Brion.}
{``Lectures on the geometry of flag varieties.''}
{\it Topics in cohomological studies of algebraic varieties},
{Trends Math.},
{33--85},
{Birkh\"auser, Basel}
{(2005)}.

%\bibitem[BBL]{BBL}
%B. Brubaker, D. Bump, and A. Licata.
%``Whittaker functions and Demazure operators.''
%{\it J. Number Theory} {\bf 146} (2015), 41--68.

\bibitem[Ch1]{Ch1}
I. Cherednik. 
``Intertwining operators of double affine Hecke algebras.''
{\it Selecta Math. (N.S.)} {\bf 3} (1997), no. 4, 459--495.

\bibitem[Ch2]{Ch2}
I. Cherednik.
Double affine Hecke algebras. 
{\it London Mathematical Society Lecture Note Series}, {\bf 319}. Cambridge University Press, Cambridge (2005). 
  
\bibitem[Deo]{Deo}
{V. V. Deodhar.}
{``On some geometric aspects of {B}ruhat orderings. {I}. {A} finer decomposition of {B}ruhat cells.''}
{\it Invent. Math.}
{\bf 79}
{(1985)},
3,
{499--511}.

\bibitem[Dy]{Dy}
M. J. Dyer.
``Hecke algebras and shellings of Bruhat intervals. II. Twisted Bruhat orders.''
Kazhdan-Lusztig theory and related topics (Chicago, IL, 1989), 141--165. 
{\it Contemp. Math.} {\bf 139}, Amer. Math. Soc., Providence, RI (1992). 
		
\bibitem[Dy2]{Dy2}
M. J. Dyer.
``Modules for the dual nil Hecke ring.''
{\tt http://www3.nd.edu/~dyer/papers/nilhecke.pdf}

\bibitem[GL]{Gaussent-Littelmann}
{S. Gaussent, and P. Littelmann.}
{``L{S} galleries, the path model, and {MV} cycles.''}
{\it Duke Math. J.}
{\bf 127}
{(2005)},
{1},
{35--88}.

\bibitem[GR1]{GR1}
{S. Gaussent, and G. Rousseau.}
{``Kac-{M}oody groups, hovels and {L}ittelmann paths.''}
{\it Ann. Inst. Fourier (Grenoble)},
{\bf 58},
{(2008)},
{no. 7},
{2605--2657}.
		


\bibitem[GR2]{GR2}
{S. Gaussent, and G. Rousseau.}
``Spherical {H}ecke algebras for {K}ac-{M}oody groups over local fields.''
{\it Ann. of Math. (2)},
{\bf 180},
{(2014)},
{no. 3},
{1051--1087}.
	
\bibitem[Go]{Gobet}
T. Gobet.
``Thomas Twisted filtrations of Soergel bimodules and linear Rouquier complexes.''
{\it J. Algebra} {\bf 484} (2017), 275--309.

\bibitem[G\"o]{Goertz}
U. G\"{o}rtz.
``Alcove walks and nearby cycles on affine flag manifolds.''
{\it J. Algebraic Combin.} {\bf 26} (2007), no. 4, 415--430.

\bibitem[Ion]{Ion}
B. Ion.
``A weight multiplicity formula for Demazure modules.''
{\it Int. Math. Res. Not.} (2005), no. 5, 311--323. 

%\bibitem[Ion2]{Ion2}
%Ion, B.
%``Nonsymmetric Macdonald polynomials and matrix coefficients for unramified principal series.''
%{\it Adv. Math.}, {\bf 201} (2006), no. 1, 36--62.

\bibitem[KL]{KL}
D. Kazhdan and G. Lusztig.
``Representations of Coxeter groups and Hecke algebras.''
{\it Invent. Math.} {\bf 53} (1979), no. 2, 165--184. 

\bibitem[Kum]{Kum}
S. Kumar.
Kac-Moody groups, their flag varieties and representation theory.
{\it Progress in Mathematics}, {\bf 204}. Birkh\"auser Boston, Inc., Boston, MA, 2002.

\bibitem[Lus]{Lus}
G. Lusztig.
``Affine Hecke algebras and their graded version.''
{\it J. Amer. Math. Soc.} {\bf 2} (1989), no. 3, 599--635. 

\bibitem[Mac]{Mac}
 I. G. Macdonald.
Affine Hecke algebras and orthogonal polynomials. 
{\it Cambridge Tracts in Mathematics}, {\bf 157}.
Cambridge University Press, Cambridge, 2003.

\bibitem[M]{M}
D. Muthiah.
``On Iwahori-Hecke Algebras for $p$-adic Loop Groups: Double Coset Basis and Bruhat Order.''
{\it Amer. J. Math.}, to appear.

\bibitem[MO]{MO}
D. Muthiah, and D. Orr.
``On the double-affine Bruhat order: the $\varepsilon=1$ conjecture and classification of covers in ADE type.''
arXiv:1609.03653
 
%\bibitem[OS]{OS}
%Orr, D. and Shimozono, M.
%``Specializations of nonsymmetric Macdonald-Koornwinder polynomials.''
%{\it J. Algebraic Combin.}, to appear.

\bibitem[PRS]{PRS}
J. Parkinson, A. Ram, and C. Schwer.
``Combinatorics in affine flag varieties.''
{\it J. Algebra} {\bf 321} (2009), no. 11, 3469--3493.

\bibitem[Ram]{Ram}
{A. Ram.}
{``Alcove walks, {H}ecke algebras, spherical functions, crystals and column strict tableaux.''}
{\it Pure Appl. Math. Q.}
{\bf 2}
{(2006)},
{4, part 2},
{963--1013}.

\bibitem[RY]{RY}
A. Ram, and M. Yip.
``A combinatorial formula for Macdonald polynomials.''
{\it Adv. Math.} {\bf 226} (2011), no. 1, 309--331.

\bibitem[Sch]{Sch}
C. Schwer.
``Galleries, Hall-Littlewood polynomials, and structure constants of the spherical Hecke algebra.''
{\it Int. Math. Res. Not.} (2006), Art. ID 75395, 31 pp. 

\end{thebibliography}
\end{document}